\documentclass{amsart}
\usepackage{amsmath, amsthm, amssymb, amscd, mathrsfs, eucal, epsfig}
\usepackage{pinlabel}

\DeclareMathOperator{\SL}{\mathrm{SL}}

\DeclareMathOperator{\GL}{\mathrm{GL}}

\input xy
\xyoption{all}

\usepackage{hyperref}

\begin{document}

\newtheorem{theorem}{Theorem}[subsection]
\newtheorem{lemma}[theorem]{Lemma}
\newtheorem{corollary}[theorem]{Corollary}
\newtheorem{conjecture}[theorem]{Conjecture}
\newtheorem{proposition}[theorem]{Proposition}
\newtheorem{question}[theorem]{Question}
\newtheorem{problem}[theorem]{Problem}
\newtheorem*{PL_LO}{PL locally indicable Theorem~\ref{theorem:pillow}}
\newtheorem*{PL_distortion}{PL distortion Corollary~\ref{corollary:distorted}}
\newtheorem*{PL_CO}{PL circularly orderable Theorem~\ref{theorem:PL_ropes_co}}
\newtheorem*{claim}{Claim}
\newtheorem*{criterion}{Criterion}
\theoremstyle{definition}
\newtheorem{definition}[theorem]{Definition}
\newtheorem{construction}[theorem]{Construction}
\newtheorem{notation}[theorem]{Notation}
\newtheorem{convention}[theorem]{Convention}
\newtheorem*{warning}{Warning}

\theoremstyle{remark}
\newtheorem{remark}[theorem]{Remark}
\newtheorem{example}[theorem]{Example}
\newtheorem*{case}{Case}

\def\R{\mathbb R}
\def\Z{\mathbb Z}
\def\RP{\mathbb{RP}}
\def\PL{\textnormal{PL}}
\def\GL{\textnormal{GL}}
\def\Homeo{\textnormal{Homeo}}
\def\Diff{\textnormal{Diff}}
\def\fix{\textnormal{fix}}
\def\fro{\textnormal{fro}}
\def\Id{\textnormal{Id}}

\newcommand{\marginal}[1]{\marginpar{\tiny #1}}

\title{Groups of PL homeomorphisms of cubes}
\author{Danny Calegari}
\address{Department of Mathematics \\ University of Chicago \\
Chicago, Illinois, 60637}
\email{dannyc@math.uchicago.edu}
\author{Dale Rolfsen}
\address{Department of Mathematics \\ University of British Columbia \\ Vancouver \\ Canada}
\email{rolfsen@math.ubc.ca}
\dedicatory{D\'edi\'e \`a Michel Boileau sur son soixanti\`eme anniversaire. Sant\'e!}

\date{\today}

\begin{abstract}
We study algebraic properties of 
groups of PL or smooth homeomorphisms of unit cubes in any dimension, fixed pointwise on the boundary,
and more generally PL or smooth groups acting on manifolds and fixing pointwise a submanifold
of codimension 1 (resp. codimension 2), and show that such groups are locally indicable
(resp. circularly orderable). We also give many examples of interesting groups that can act,
and discuss some other algebraic constraints that such groups must satisfy, including the fact that
a group of PL homeomorphisms of the $n$-cube (fixed pointwise on the boundary) contains no
elements that are more than exponentially distorted.
\end{abstract}

\maketitle

\section{Introduction}

We are concerned in this paper with algebraic properties of the
group of $\PL$ homeomorphisms of a $\PL$ manifold, fixed on some $\PL$ 
submanifold (usually of codimension 1 or 2, for instance the boundary) 
and some of its subgroups (usually those preserving some structure).
The most important case is the group of $\PL$ homeomorphisms of $I^n$ fixed
pointwise on $\partial I^n$; hence these are ``groups of PL homeomorphisms of the ($n$-)cube''.

The {\em algebraic} study of transformation groups (often in low dimension, or
preserving some extra structure such as a symplectic or complex structure)
has recently seen a lot of activity; however, much of this activity has been
confined to the smooth category. It is striking that many of these results 
can be transplanted to the $\PL$ category. This
interest is further strengthened by the possibility of working in the $\PL$
category over (real) algebraic rings or fields (this possibility has already
been exploited in dimension 1, in the groups $F$ and $T$ of Richard Thompson).

Many theorems we prove have analogs in the ($C^1$) smooth category, and we 
usually give proofs of such theorems for comparison where they are not 
already available in the literature.

\subsection{Statement of results}

The main algebraic properties of our groups that we establish are
{\em left orderability} (in fact, local indicability) and {\em controlled distortion}.
These are algebraic properties which at the same time are readily compared with 
geometric or topological properties of a group action.
For example, the property of left orderability for a countable group is {\em equivalent} to
the existence of a faithful action on $(I,\partial I)$ by homeomorphisms. As another example,
control of (algebraic) distortion has recently been used by Hurtado \cite{Hurtado} to
prove strong rigidity results for homomorphisms between various (smooth) transformation groups.

In \S~\ref{section:background} we state for the convenience of the reader standard
definitions and results from the theory of left-orderable groups.

In \S~\ref{section:PL_groups} we begin our analysis of PL groups of homeomorphisms of
manifolds, fixed pointwise on a codimension 1 submanifold. If $M$ is a PL manifold and $K$ a
submanifold, we denote by $\PL_+(M,K)$ the group of orientation-preserving PL homeomorphisms of $M$
fixed pointwise on $K$. We similarly denote by $\Diff_+(M,K)$ and $\Homeo_+(M,K)$ the
groups of diffeomorphisms (resp. homeomorphisms) of a smooth (resp. topological) manifold $M$,
fixed pointwise on a smooth (resp. topological) submanifold $K$ of codimension 1. 
The first main theorem of this section is:

\begin{PL_LO}
Let $M$ be an $n$ dimensional connected $\PL$ manifold, and let $K$ be a
nonempty closed $\PL$ submanifold of codimension at most 1.
Then the group $\PL_+(M,K)$ is left orderable; in fact, it is locally
indicable.
\end{PL_LO}

We give some basic constructions of interesting groups of PL homeomorphisms of cubes (of dimension at 
least 2) fixed pointwise on the boundary, including examples of free subgroups, groups with
infinite dimensional spaces of (dynamically defined) quasimorphisms, and right-angled Artin groups;
in fact, the method of Funar, together with a recent result of Kim and Koberda shows that every RAAG embeds in $\PL(I^2,\partial I^2)$. 

We conclude this section by studying distortion in $\PL(I^m,\partial I^m)$, and prove

\begin{PL_distortion}
Every element of $\PL(I^m,\partial I^m)-\Id$ is at most exponentially distorted.
\end{PL_distortion}

This lets us easily construct explicit examples of locally indicable groups which are not isomorphic to
subgroups of $\PL(I^m,\partial I^m)$ for any $m$ (for example, iterated HNN extensions).

In \S~\ref{section:smooth} we prove Theorem~\ref{theorem:cilo}, the analog of
Theorem~\ref{theorem:pillow} for groups of $C^1$ smooth diffeomorphisms. The main purpose of
this section is to compare and contrast the methods of proof in the PL and smooth categories.

In \S~\ref{section:biorderable} we sharpen our focus to the question of bi-orderability (i.e.\/
orders which are invariant under both left and right multiplication), for
transformation groups acting on cubes in various dimensions and with differing degrees of
analytic control. In dimensions bigger than 1, none of the groups we study are bi-orderable.
In dimension 1, the groups $\PL(I,\partial I)$ and $\Diff^\omega(I,\partial I)$ are bi-orderable.

In \S~\ref{section:circular} we consider groups acting on manifolds and fixing pointwise submanifolds
(informally, ``knots'') of codimension two. Here we are able to bootstrap our results from
previous sections to show the following:

\begin{PL_CO}
Let $M$ be an $n$ dimensional connected $\PL$ orientable manifold, and let $K$ be a
nonempty $n-2$ dimensional closed submanifold. Then the group $\PL_+(M,K)$ is
circularly orderable.
\end{PL_CO}

An analog for $C^1$ diffeomorphisms is also proved. An interesting application is to the
case that $M$ is a 3-manifold, and $K$ is a hyperbolic knot (i.e.\/ a knot with a hyperbolic
complement). In this case, if $G_0(M,K)$ denotes the group of PL homeomorphisms 
(or $C^1$ diffeomorphisms) isotopic to the identity and taking $K$ to itself (but {\em not} necessarily
fixing it pointwise) then $G_0(M,K)$ is left-orderable. This depends on theorems of Hatcher and Ivanov
on the topology of $G_0(M,K)$ as a topological group.

Finally, in \S~\ref{section:finite} we briefly discuss groups of homeomorphisms with no analytic
restrictions. Our main point is how little is known in this generality in dimension $>1$;
in particular, it is not even known if the group $\Homeo(I^2,\partial I^2)$ is left-orderable.
We also make the observation that the groups $\Homeo(I^n,\partial I^n)$ are torsion-free for all
$n$; this follows immediately from Smith theory, but the result does not seem to be well-known to
people working on left-orderable groups, so we believe it is useful to include
an argument here.

\subsection{Acknowledgement}

We would like to thank Andr\'es Navas and Amie Wilkinson for some helpful discussions
about this material. Danny Calegari was supported by NSF grant DMS 1005246.
Dale Rolfsen was supported by a grant from the Canadian Natural Sciences and Engineering Research
Council.

\section{Left orderable groups}\label{section:background}

This section contains a very brief summary of some standard results in the theory of
left orderable groups. These results are collected here for the convenience of the reader.
For proofs, see e.g.\/ \cite{Calegari_foliations}, Chapter~2.

\begin{definition}[Left orderable]\label{definition:left_orderable}
A group $G$ is {\em left orderable} (usually abbreviated to LO) if there is a total order
$\prec$ on $G$ so that for all $f,g,h \in G$ the relation $g \prec h$ holds if and only if
$fg\prec fh$ holds.
\end{definition}

\begin{lemma}\label{lemma:ses}
If there is a short exact sequence $0 \to K \to G \to H \to 0$ and both $K$ and $H$ are left
orderable, then $G$ is left orderable.
\end{lemma}

\begin{lemma}\label{lemma:LLO}
A group if left orderable if and only if every finitely generated subgroup is left orderable.
\end{lemma}

\begin{lemma}\label{lemma:LO_action}
A countable group $G$ is left orderable if and only if it isomorphic to 
a subgroup of $\Homeo(I,\partial I)$.
\end{lemma}

\begin{definition}[Locally indicable]\label{defininion:locally_indicable}
A group $G$ is {\em locally indicable} if for every finitely generated nontrivial subgroup
$H$ of $G$ there is a surjective homomorphism from $H$ to $\Z$.
\end{definition}

\begin{theorem}[Burns-Hale, \cite{Burns_Hale}]\label{theorem:Burns_Hale}
Every locally indicable group is left orderable.
\end{theorem}

\section{\texorpdfstring{$\PL$}{PL} group actions on cubes}\label{section:PL_groups}

\subsection{Definitions}

We assume the reader is familiar with the concept of a $\PL$ homeomorphism between compact
polyhedra in $\R^n$, namely one for which the domain can be subdivided into {\em finitely many}
linear simplices so that the restriction of the homeomorphism to each simplex is an
affine linear homeomorphism to its image. We say that the simplices in the domain are {\em linear
for $f$}. A {\em $\PL$ manifold} is one with charts modeled on $\R^n$ and
transition functions which are the restrictions of $\PL$ homeomorphisms.

We define the {\em dimension} of a (linear) polyhedron to be the maximum of the dimensions of
the simplices making it up, in any decomposition into simplices.

\begin{notation}
Let $M$ be a $\PL$ manifold (possibly with boundary) and let $K$ be a
closed $\PL$ submanifold of $M$. We denote by
$\PL(M,K)$ the group of $\PL$ self-homeomorphisms of $M$ which are fixed
pointwise on $K$. If $M$ is orientable, we denote by $\PL_+(M,K)$ the subgroup
of orientation-preserving $\PL$ self-homeomorphisms.
\end{notation}

The main example of interest is the following:

\begin{notation}
We denote by $I^n$ the cube $[-1,1]^n$ in Euclidean space with its standard $\PL$ structure.
Note with this convention that $0$ is a point in the {\em interior} of $I^n$. We denote the
boundary of $I^n$ by $\partial I^n$.
\end{notation}

So we denote by $\PL(I^n,\partial I^n)$ the group of $\PL$ self-homeomorphisms of the unit cube
$I^n$ in $\R^n$ which are fixed pointwise on the boundary. If $\omega$ denotes
the standard (Lesbesgue) volume form on $I^n$, we denote by $\PL_\omega(I^n,\partial I^n)$
the subgroup of $\PL(I^n,\partial I^n)$ preserving $\omega$.

\subsection{Transformation groups as discrete groups}

Let $G$ be a transformation group --- i.e.\/ a 
group of homeomorphisms of some topological space $X$. It is often useful to
endow $G$ with a topology compatible with the action; for instance, 
the {\em compact-open topology}. If we denote $G^\delta$ as the same group but
with the discrete topology, the identity homomorphism $G^\delta \to G$ is
a continuous map of topological groups, and induces maps on cohomology
$H^*(BG;R) \to H^*(BG^\delta;R) = H^*(K(G,1);R)$ for any coefficient module $R$.
Thus one interesting source of algebraic invariants of the discrete group $G$
arise by thinking of its homotopy type as a topological group.

However, the groups of most interest to us in this paper
are not very interesting as (homotopy types of) topological spaces:

\begin{proposition}[Alexander trick]\label{proposition:Alexander_trick}
The groups $\Homeo(I^n,\partial I^n)$, $\PL(I^n,\partial I^n)$, $\PL_\omega(I^n,\partial I^n)$
are all contractible in the compact-open topology.
\end{proposition}
\begin{proof}
Given $f:I^n \to I^n$ fixed pointwise on $\partial I^n$, define $f_t:I^n \to I^n$
by 
$$f_t(x):=\begin{cases}
tf(x/t) & \text{ if } 0\le \|x\| \le t \\
x & \text{ if } t \le \|x\| \le 1
\end{cases}$$
where $\|\cdot\|$ denotes the $L^\infty$ norm on $I^n$.
Then $f_0=f$ and $f_1 = \Id$, and the assignment $f \to f_t$ defines a deformation
retraction of any one of the groups in question to the identity.
\end{proof}

\subsection{Left orderability}\label{PL_LO_subsection}

In this section we prove the left orderability of certain groups of $\PL$
homeomorphisms.

The purpose of this section is to develop the tools to prove the following theorem:

\begin{theorem}[Locally indicable]\label{theorem:pillow}
Let $M$ be an $n$ dimensional connected $\PL$ manifold, and let $K$ be a
nonempty closed $\PL$ submanifold of codimension at most 1.
Then the group $\PL_+(M,K)$ is left orderable; in fact, it is locally
indicable.
\end{theorem}

An important special case is $M=I^n$ and $K=\partial I^n$.

First we introduce some notation and structure. Recall that if $X$ is a closed subset of $Y$,
the {\em frontier} of $X$ in $Y$ is the intersection of $X$ with closure of $Y-X$.

\begin{notation}
Let $f \in \PL(M,K)$. The {\em fixed point set} of $f$ is denoted $\fix(f)$ and
the {\em frontier} of $\fix(f)$ is denoted $\fro(f)$. Similarly, if 
$G$ is a subgroup of $\PL(M,K)$, 
the {\em fixed point set} of $G$ is denoted $\fix(G)$. The
{\em frontier of $\fix(G)$} is denoted $\fro(G)$. 
\end{notation}

\begin{lemma}\label{lemma:fix_and_frontier}
Let $M$ be an $n$ dimensional connected $\PL$ manifold, and let $K$ be a
nonempty $n-1$ dimensional closed $\PL$ submanifold.
Let $G$ be a nontrivial finitely generated subgroup of $\PL(M,K)$.
\begin{enumerate}
\item{$\fix(G)$ is a closed polyhedron of dimension at least $n-1$.}
\item{$\fro(G)$ is a polyhedron of dimension $n-1$.}
\item{If $g_1,\cdots,g_m$ is a finite generating set for $G$, then $\fro(G) \subset \cup_i \fro(g_i)$.}
\end{enumerate}
\end{lemma}
\begin{proof}
Let $g_1,\cdots,g_m$ be a finite generating set for $G$. 

(1) We have $\fix(G) = \cap_i \fix(g_i)$.
Since each $\fix(g_i)$ is a closed polyhedron containing $K$, so is $\fix(G)$.

(2) If $\fix(G)$ has no interior, then its complement is open and dense, so $\fro(G)=\fix(G)$.
Otherwise, $\fix(G)$ has some interior, and $\fro(G)$ separates some point in the interior from
some point in $M - \fix(G)$. Since $M$ is connected, $\fro(G)$ has dimension at least $n-1$.

(3) Every point $p$ in $\fro(G)$ is in $\fix(g_i)$ for all $i$, and for some $i$ there are points
arbitrarily near $p$ moved nontrivially by $g_i$. Thus $p \in \fro(g_i)$ for this $i$.
\end{proof}

\begin{definition}
Let $G$ be a subgroup of $\PL(M,K)$.
The action of $G$ is {\em semilinear} at a point $p$ if there is some codimension 1
plane $\pi$ through $p$ and a convex open neighborhood $U$ of $p$ so that the restriction of
$G$ is linear on both components of $U - (\pi \cap U)$. We call $\pi$ the {\em dividing plane}.
\end{definition}

Note with this definition that a linear action at a point is semilinear. Note also that 
if $G$ is semilinear at a point $p$ but not linear there, the dividing plane is unique.

\begin{lemma}\label{lemma:semi_linear}
Let $M$ be an $n$ dimensional connected $\PL$ manifold, and let $K$ be a
nonempty $n-1$ dimensional closed $\PL$ submanifold.
Let $G$ be a finitely generated subgroup of $\PL(M,K)$.
Then for every $n-1$ dimensional linear simplex $\sigma$ in $\fro(G)$, there is an open and
dense subset of $\sigma$ where $G$ is semilinear, and the dividing plane is tangent to $\sigma$.
\end{lemma}
\begin{proof}
Fix a finite generating set $g_1,\cdots,g_m$ for $G$, and recall from 
Lemma~\ref{lemma:fix_and_frontier} bullet (3) that $\fro(G) \subset \cup_i \fro(g_i)$.
Let $\sigma$ be an $n-1$ dimensional simplex in $\fro(G)$, and suppose $\sigma \cap \fro(g_i)$ has
full dimension. Let $\sigma_i = \sigma \cap \fro(g_i)$.
Associated to $g_i$ there is a decomposition of $M$ into linear simplices; away from
the $n-2$ skeleton of this decomposition, $g_i$ acts semilinearly.
So $g_i$ acts semilinearly on the complement of an $n-2$ dimensional polyhedral subset $b_i$
of $\sigma_i$. For each point $p \in \sigma_i - b_i$ the action of $g_i$ on the tangent plane
to $\sigma_i$ is the identity (since $\sigma_i$ is fixed by $g_i$) and therefore we can
take the dividing plane to be equal to this tangent plane. For, either the action at $p$ is
semilinear but not linear (in which case any codimension 1 plane on which the action is linear is
the unique dividing plane), or the action at $p$ is linear, in which case any codimension 1 plane
can be taking to be the dividing plane. Thus the dividing plane of $g_i$ is tangent to
$\sigma_i$ along $\sigma_i - b_i$. Furthermore, since $\sigma - \sigma_i$ is in $\fix(g_i)$ but not $\fro(g_i)$,
the element $g_i$ acts trivially in a neighborhood of each point of $\sigma - \sigma_i$.

It follows that each $g_i$ acts semilinearly at each point of $\sigma - \cup_i b_i$ with
dividing plane tangent to $\sigma$, and therefore all of $G$ acts semilinearly at each point
of $\sigma - \cup_i b_i$ with dividing plane tangent to $\sigma$.
\end{proof}

\begin{lemma}\label{lemma:linear_stabilizer}
The (pointwise) stabilizer in $\GL(n,\R)$ of an $m$ dimensional subspace $\pi$ of $\R^n$
is isomorphic to $\R^{m(n-m)} \rtimes \GL(m,\R)$ and the subgroup preserving
orientation is isomorphic to $\R^{m(n-m)} \rtimes \GL^+(m,\R)$.
\end{lemma}
\begin{proof}
The subgroup of $\GL(n,\R)$ fixing $\pi$ can be conjugated into the set of 
matrices of the form $\left(\begin{smallmatrix} \Id & V \\ 0 & A \end{smallmatrix}\right)$
where $\Id$ is the identity in $\GL(n-m,\R)$, where $V$ is an arbitrary $(n-m)\times m$ matrix, 
and where $A \in \GL(m,\R)$. The proof follows.
\end{proof}

\begin{example}[Codimension 1 stabilizer]\label{example:codimension_1}
In the special case $m=n-1$ the stabilizer is isomorphic to $\R^{n-1}\rtimes \R^*$,
and the orientation preserving subgroup is isomorphic to $\R^{n-1} \rtimes \R^+$ 
which is an extension of the locally indicable group $\R^+$ by the locally 
indicable group $\R^{n-1}$, and is therefore locally indicable.
\end{example}

We are now ready to give the proof of Theorem~\ref{theorem:pillow}:

\begin{proof}
By Theorem~\ref{theorem:Burns_Hale} it suffices to show that every nontrivial
finitely generated subgroup of
$\PL_+(M,K)$ surjects onto $\Z$. Let $G$ be such a subgroup, and let $p$ be a point in
a codimension 1 simplex in $\fro(G)$ where $G$ acts semilinearly. Since $p$ is in $\fro(G)$,
some $g_i$ acts nontrivially on one side of $\fro(G)$, so there is a {\em nontrivial} homomorphism
from $G$ to $\R^{n-1}\rtimes \R^+$ which is locally indicable, and therefore $G$ surjects
onto $\Z$. Since $G$ was arbitrary, the theorem is proved.
\end{proof}

\subsection{Free subgroups of \texorpdfstring{$\PL(I^n,\partial I^n)$}{PLIn}}

In this section we discuss algebraic properties of the groups $\PL(I^n,\partial I^n)$
and their subgroups. The case $n=1$ is well-studied, and one knows the following striking theorem:

\begin{theorem}[Brin-Squier, \cite{Brin_Squier}]\label{theorem:Brin_Squier}
The group $\PL(I,\partial I)$ obeys no law, but does not contain a nonabelian free subgroup.
\end{theorem}

On the other hand, we will shortly see that the group $\PL(I^n,\partial I^n)$ contains nonabelian
free subgroups for all $n>1$.

\begin{lemma}\label{lemma:free_in_PL_2}
The group $\PL(I^2,\partial I^2)$ contains nonabelian free subgroups.
\end{lemma}
\begin{proof}
For any $M \in \GL(2,\R)$ there is some $g_M$ in $\PL(I^2,\partial I^2)$ which fixes $0$, which
is linear near $0$, and which satisfies $dg_M(0) = M$. Let $G$ be the group generated by the
$g_M$. Then $G$ fixes $0$ and is linear there, so there is a surjective homomorphism 
$G \to \GL(2,\R)$. If $F$ is a nonabelian free subgroup of $\GL(2,\R)$, there is a section from
$F$ to $G$ whose image is a free subgroup of $G$.
\end{proof}

In fact we will shortly see (Theorem~\ref{theorem:RAAG}) 
that every right-angled Artin group embeds in $\PL(I^2,\partial I^2)$.

\begin{lemma}\label{lemma:suspension_of_PL}
For all $n$ there is an injective homomorphism 
$$S:\PL(I^n,\partial I^n) \to \PL(I^{n+1},\partial I^{n+1})$$
called the {\em suspension homomorphism}.
\end{lemma}
\begin{proof}
We include $I^n$ into $I^{n+1}$ as the subset with last coordinate equal to $0$. The
{\em bipyramid} $P^{n+1}$ is the convex hull of $I^n$ and the two points $(0,0,\cdots,0,\pm 1)$.
For every simplex $\sigma$ in $I^n$ there are two simplices $S^{\pm}\sigma$ of one dimension higher,
obtained by coning $\sigma$ to $(0,0,\cdots,0,\pm 1)$. If $f$ is a $\PL$ homeomorphism of
$I^n$ fixed on the boundary, and $\sigma$ is a simplex in $I^n$ linear for $f$, then there is
a unique $\PL$ homeomorphism $Sf$ of $P^{n+1}$ which takes each $S^{\pm}\sigma$ linearly to
$S^{\pm}f(\sigma)$. The map $f \to Sf$ defines an injective homomorphism 
$\PL(I^n,\partial I^n) \to \PL(P^{n+1},\partial P^{n+1})$. Then any element of
$\PL(P^{n+1},\partial P^{n+1})$ can be extended by the identity on $I^{n+1} - P^{n+1}$ to an
element of $\PL(I^{n+1},\partial I^{n+1})$.
\end{proof}

\begin{corollary}\label{corollary:free_subgroup}
The group $\PL(I^n,\partial I^n)$ contains a nonabelian free subgroup for all $n>1$.
\end{corollary}
\begin{proof}
By Lemma~\ref{lemma:free_in_PL_2}, Lemma~\ref{lemma:suspension_of_PL} and induction.
\end{proof}

Actually, we will see many more constructions of free subgroups of $\PL(I^n,\partial I^n)$
in the sequel.

\subsection{Area preserving subgroups}

We are interested in the following subgroup of $\PL(I^2,\partial I^2)$.

\begin{definition}
Let $\omega$ denote the (standard) area form on $I^2$. Let $\PL_\omega(I^2,\partial I^2)$
denote the subgroup of $\PL(I^2,\partial I^2)$ consisting of transformations which preserve $\omega$.
\end{definition}

The group $\PL_\omega(I^2,\partial I^2)$ contains many interesting subgroups, and has a
rich algebraic structure. We give some indications of this.

\begin{example}[Dehn twist]\label{example:Dehn_twist}
Figure~\ref{elementary_twist} depicts an area-preserving PL homeomorphism of a square,
which preserves the foliation by concentric squares, and (for suitable choices
of edge lengths) whose 12th power is a Dehn twist (in fact, this transformation is contained
in a 1-parameter subgroup of $\PL_\omega(I^2,\partial I^2)$ consisting of powers of a Dehn
twist at discrete times).
\begin{figure}[htpb]
\labellist
\small\hair 2pt
\pinlabel $\xrightarrow{\hspace*{1cm}}$ at 300 125
\endlabellist
\centering
\includegraphics[scale=0.5]{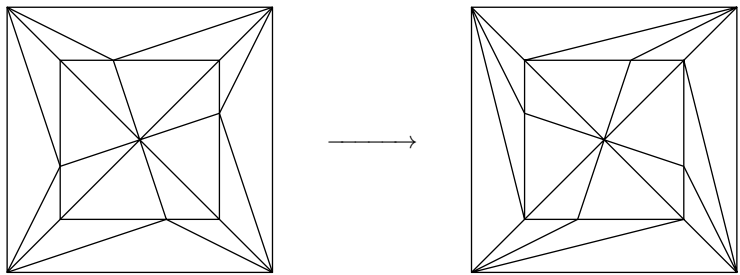}
\caption{A 12th root of a PL Dehn twist}\label{elementary_twist}
\end{figure}
\end{example}

The following theorem is due to Gratza \cite{Gratza}; in fact he proved
the analogous statement for {\em piecewise linear symplectic} automorphisms of
$I^{2n}$. 

\begin{theorem}[Gratza]\label{theorem:density}
The group $\PL_\omega(I^2,\partial I^2)$ is dense in $\Diff_\omega(I^2,\partial I^2)$,
in the $C^0$ topology.
\end{theorem}

Gratza's theorem has the following application: it can be used to certify that
$\PL_\omega(I^2,\partial I^2)$ admits an infinite dimensional space of
nontrivial {\em quasimorphisms} (i.e.\/ quasimorphisms which are not homomorphisms). 
By contrast, no subgroup of $\PL(I,\partial I)$
admits a nontrivial quasimorphism, by \cite{Calegari_PL}.
For an introduction to the theory of quasimorphisms and its relation to stable
commutator length, see \cite{Calegari_scl}.

\begin{definition}
Let $G$ be a group. A {\em homogeneous quasimorphism} on $G$ is a function
$\phi:G \to \R$ satisfying the following properties:
\begin{enumerate}
\item{(homogeneity) for any $g\in G$ and any $n \in \Z$ we have $\phi(g^n) = n\phi(g)$; and}
\item{(quasimorphism) there is a least non-negative number $D(\phi)$ (called
the {\em defect}) so that for all $g,h \in G$ there is an inequality
$$|\phi(gh) - \phi(g) - \phi(h)| \le D(\phi)$$}
\end{enumerate}
The (real vector) space of homogeneous quasimorphisms on $G$ is denoted $Q(G)$.
\end{definition}

A function satisfying the second condition but not the first is said to be
(simply) a {\em quasimorphism}. If $\phi:G \to \R$ is any quasimorphism, the function
$\bar{\phi}:G \to \R$ defined by $\bar{\phi}(g)=\lim_{n \to \infty} \phi(g^n)/n$ is
a homogeneous quasimorphism, and satisfies $\bar{\phi} - \phi \le D(\phi)$.
This operation is called {\em homogenization} of quasimorphisms.

A homogeneous quasimorphism has defect $0$ if and only if it is a homomorphism to
$G$. Thus $D$ descends to a {\em norm} on the quotient space $Q(G)/H^1(G;\R)$.
It is a fact that $Q/H^1$ with the defect norm is a {\em Banach space}.

It is known that $Q/H^1$ vanishes on any subgroup $G$ of $\PL(I,\partial I)$.
By contrast, we show that the subgroup $\PL_\omega(I^2,\partial I^2)$ 
admits an infinite dimensional $Q/H^1$.

The proof is by explicit construction, and based on a general method due to
Gambaudo-Ghys \cite{Gambaudo_Ghys}. The construction is as follows.
First, fix some $n$ and let $B_n$ denote the braid group on $n$ strands, and
fix $n$ distinct points $x_1^0,\cdots,x_n^0$ in
the interior of $I^2$.
Let $\mu:B_n \to \R$ be any function. 

Now, since $\Homeo(I^2,\partial I^2)$ is path-connected and simply-connected,
for any $g \in \PL_\omega(I^2,\partial I^2)$ there is a unique homotopy class of
isotopy $g_t$ from $g$ to $\Id$.

For any (generic) $n$-tuple of distinct points
$x_1,\cdots,x_n$ in the interior of $I^2$, and any $g \in \PL_\omega(I^2,\partial I^2)$,
let $\gamma(g;x_1,\cdots,x_n) \in B_n$ be the braid obtained by first moving the
$x_i^0$ in a straight line to the $x_i$, then composing with the isotopy $g_t$ from
$x_i$ to $g(x_i)$, then finally moving the $g(x_i)$ in a straight line back to the $x_i^0$
(note that $\gamma$ does not depend on the choice of $g_t$, since $\Homeo(I^2,\partial I^2)$
is simply-connected). 

Now we define
$$\Phi_\mu(g) = \int_{I^2\times \cdots \times I^2} \mu(\gamma(g;x_1,\cdots,x_n))d\omega(x_1)d\omega(x_2)\cdots d\omega(x_n)$$
where the integral is taken over the subset of the product of $n$ copies of (the interior of)
$I^2$ consisting of distinct $n$-tuples of points where $\gamma$ is well-defined.

\begin{lemma}\label{lemma:quasimorphism}
Suppose $\mu$ is a quasimorphism on $B_n$. Then $\Phi_\mu$ is a quasimorphism on
$\PL_\omega(I^2,\partial I^2)$
\end{lemma}
\begin{proof}
Changing the vector $x^0$ to a new vector $y^0$ changes $\gamma$ by multiplication by
one of finitely many elements of $B_n$; since $\mu$ is a quasimorphism, this changes
$\Phi_\mu$ by a bounded amount.

For any two $g,h \in \PL_\omega(I^2,\partial I^2)$ and generic $x_1,\cdots,x_n$ we have
$$\gamma(gh;x_1,\cdots,x_n) = \gamma(h;x_1,\cdots,x_n)\gamma(g;h(x_1),\cdots,h(x_n))$$
Now integrate over $I^2\times \cdots \times I^2$ and use the fact that $\mu$ is
a quasimorphism to see that $\Phi_\mu$ is a quasimorphism.
\end{proof}

The homogenization $\overline{\Phi}_\mu$ is a homogeneous quasimorphism, and one may
check that it is nontrivial whenever $\overline{\mu}$ is nontrivial on $B_n$.
In fact, it is easiest to check this for the special case of homogeneous quasimorphisms
on $B_n$ that vanish on reducible elements. Such quasimorphisms are plentiful;
for example, any ``counting quasimorphism'' arising from a weakly properly
discontinuous action of $B_n$ on a hyperbolic simplicial complex.
See \cite{Calegari_scl}, \S~3.6 for more details.

\begin{lemma}
Let $\overline{\mu}$ be nontrivial on $B_n$ and vanish on all reducible elements. 
Then $\overline{\Phi}_\mu$ is
nontrivial on $\PL_\omega(I^2,\partial I^2)$.
\end{lemma}
\begin{proof}
For $1\le i\le n$ let $R_i$ be the rectangle with lower left corner
$(\frac {i-1}{n}+\epsilon,\epsilon)$ and upper right corner $(\frac {i}{n} -\epsilon,1-\epsilon)$.
For any braid $b \in B_n$ we can build an area-preserving
homeomorphism $\phi_b$ which permutes
the $R_i$, taking each $R_i$ to its image by a translation, and which performs the
conjugacy class of the braid $b$ on each $n$ tuple of points of the form
$(x,x+1,\cdots,x+n-1)$ for $x\in R_1$. 

If we are systematic about the way we extend $\phi_b$ to $I^2 - \cup_i R_i$ 
(i.e.\/ by decomposing $\phi_b$ into a product of standard ``elementary'' moves,
corresponding to the factorization of $b$ into elementary braids) we can
ensure that $\phi_b$ satisfies an estimate of the form
$|\mu(\gamma(\phi_b;x_1,\cdots,x_n))| \le C\cdot |b|\cdot D(\mu)$
where $|b|$ denotes the word length of $b$, where $D(\mu)$ denotes the defect of
$\mu$, and where $C$ is some constant depending only on $n$, but {\em not} on
$\epsilon$. It is straightforward to build such an area-preserving homeomorphism;
to see that it can be approximated by a $\PL$ area-preserving homeomorphism with
similar properties, we appeal to Gratza's Theorem~\ref{theorem:density}.

For any $n$-tuple $x:=(x_1,\cdots,x_n)$ where the $x_i$ are all contained in
distinct $R_j$, the powers of $\phi_b$ on $x$ are conjugates of the powers of
$b$. For $n$-tuples where two $x_i$ are in the same $R_j$, the powers of $\phi_b$
on $x$ are reducible. Thus we can estimate
$$\overline{\Phi}_\mu(\phi_b) = (n!/n^n)\overline{\mu}(b) + O(\epsilon)$$
Taking $b$ to be a braid on which $\overline{\mu}$ is nonzero, and
taking $\epsilon \to 0$ we obtain the desired result.
\end{proof}

In particular, $Q(\PL_\omega(I^2,\partial I^2))$ is infinite dimensional.
This should be contrasted with the 1-dimensional case, where it is shown
in \cite{Calegari_PL} that for {\em any} subgroup $G$ of $\PL(I,\partial I)$
the natural map $H^1(G) \to Q(G)$ is surjective.

\subsection{Right-angled Artin groups}

Recall that a {\em right-angled Artin group} (hereafter a RAAG) associated to a (finite
simplicial) graph $\Gamma$ is the group with one generator for each vertex of $\Gamma$,
and with the relation that two generators commute if the corresponding vertices are joined
by an edge, and with no other relations. We denote this group by $A(\Gamma)$

Funar \cite{Funar} discovered a powerful method to embed certain RAAGs in 
transformation groups. To describe Funar's method, it is convenient to
work with the complement graph $\Gamma^c$, which has the same vertex set as $\Gamma$,
and which has an edge between two distinct vertices if and only if $\Gamma$ does {\em not}
have an edge between these vertices. Thus: two vertices of $\Gamma^c$ are joined by an
edge if and only if the corresponding generators of the RAAG do not commute.

Let $S$ be a surface, and for each vertex $i$ of $\Gamma^c$,
let $\gamma_i$ be an embedded circle in $S$, chosen with the following properties:
\begin{enumerate}
\item{the various $\gamma_i$ intersect transversely; and}
\item{two circles $\gamma_i$ and $\gamma_j$ intersect if and only if the corresponding
vertices of $\Gamma^c$ are joined by an edge.}
\end{enumerate}
We say in this case that the pattern of intersection of the $\gamma_i$ {\em realizes}
$\Gamma^c$.

\begin{theorem}[Funar \cite{Funar}]\label{theorem:Funar}
With notation as above, let $\tau_i$ denote a Dehn twist in a sufficiently small
tubular neighborhood of $\gamma_i$. Then for sufficiently big $N$, the group generated
by the $\tau_i^N$ is isomorphic to $A(\Gamma)$ by an isomorphism in which the $\tau_i^N$
become the ``standard'' generators.
\end{theorem}

Kim and Koberda \cite{Kim-Koberda} applied Funar's technique to show that any 
finitely-generated RAAG embeds in $\Diff_\omega^\infty(D^2, \partial D^2)$. 
A similar argument applies to the PL case, as we now outline.  
Although not every finite graph embeds in the disk, the following theorem saves the day.

\begin{theorem}[Kim-Koberda \cite{Kim-Koberda}]\label{theorem:tree}  For each finite graph 
$\Gamma$ there exists a finite tree $T$ such that $A(\Gamma)$ is isomorphic with a subgroup of $A(T^c)$.
\end{theorem}

In fact Kim-Koberda show that $A(\Gamma)$ embeds quasi-isometrically in $A(T^c)$.

\begin{theorem}\label{theorem:RAAG}
For any $m\ge 2$, every (finitely generated) RAAG embeds in $\PL_\omega(I^m,\partial I^m)$.
\end{theorem}
\begin{proof}
By the suspension trick (Lemma~\ref{lemma:suspension_of_PL}), it suffices to consider 
the case $m = 2$.  Given $\Gamma$, apply Theorem \ref{theorem:tree} to find a tree so that 
$A(T^c)$ contains $A(\Gamma)$ as a subgroup.  Since trees are planar, 
one can find PL  curves $\gamma_i$ in $I^2$ whose pattern of intersection 
realizes $T$.  Then consider area-preserving PL Dehn twists about these 
curves and apply Theorem \ref{theorem:Funar} to see that $A(T^c)$, and
hence $A(\Gamma)$, embeds in $\PL_\omega(I^2,\partial I^2)$.
\end{proof}

\begin{corollary}
For any $m\ge 2$, if $M$ is a hyperbolic 3-manifold, some finite-index subgroup of $\pi_1(M)$
embeds in $\PL(I^m,\partial I^m)$.
\end{corollary}
\begin{proof}
Agol's proof of Wise's conjecture \cite{Agol} implies that $\pi_1(M)$ virtually
embeds in some RAAG. Now apply Theorem~\ref{theorem:RAAG}.
\end{proof}

\subsection{Distortion}

Theorem~\ref{theorem:pillow} says that every subgroup of $\PL(I^m,\partial I^m)$ is locally
indicable. So far, for $m\ge 2$ we have exhibited no other obstruction to embedding a
locally indicable group in $\PL(I^m,\partial I^m)$. In this subsection we give some more
powerful obstructions based on the notion of {\em distortion}.

\begin{definition}[Distortion]\label{definition:distortion}
Let $G$ be a group with a finite generating set $S$, and let $|\cdot|_S$ denote word
length in $G$ with respect to $S$. An element $g \in G$ is {\em distorted}
if $$\lim_{n \to \infty} \frac {|g^n|_S} {n} = 0$$
More precisely, we say $g$ has polynomial/exponential/superexponential (etc.)
distortion if 
$$\phi_g(m):=\min \lbrace n \text{ such that } |g^n|_S \ge m \rbrace$$
grows polynomially in $m$ (with degree $>1$), exponentially, superexponentially, etc.

If $G$ is an infinitely generated group, $g \in G$ is distorted if there is a finitely
generated subgroup $H$ of $G$ containing $g$ so that $g$ is distorted in $H$. It has
polynomial/exponential/superexponential (etc.) distortion in $G$ if it has such distortion
in some $H$.
\end{definition}

Note with this definition that torsion elements are (infinitely) distorted; some
authors prefer the convention that distortion elements should have infinite order.
For an introduction to the use of distortion in dynamics, see e.g. \cite{Franks_Handel},
\cite{Calegari_Freedman} or \cite{Militon}.

\begin{remark}
We remark that other kinds of distortion can be usefully studied in dynamics; for instance,
if $G$ is a group, a finitely generated subgroup $H$ is {\em distorted} in $G$ if there is
a finitely generated subgroup $K \subset G$ containing $H$ so that the inclusion of $H$ into
$K$ is not a quasi-isometric embedding (with respect to their respective word metrics). The kind of
distortion we consider here is the restriction to the case that $H$ is $\Z$. With this
more flexible definition, many interesting examples of distortion can be obtained, even in
$\PL(I,\partial I)$. For example, Wladis \cite{Wladis} shows that Thompson's group of
dyadic rational PL homeomorphisms of the interval is exponentially distorted in the (closely related)
Stein-Thompson groups.
\end{remark}

The simplest way to show that an element $g$ is undistorted in a group $G$, 
or to bound the amount of distortion, is to find a subadditive, non-negative function on $G$ whose
restriction to powers of $g$ grows at a definite rate. We give {\em two} examples of such functions
on groups of the form $\PL(I^m,\partial I^m)$, which show that elements can be at most
exponentially distorted.

\begin{definition}[Matrix norm]\label{definition:matrix_norm}
For $M \in \GL(m,\R)$ define 
$$D(M)=\sup_{i,j} |M_{ij}|$$ where the $M_{ij}$ are
the matrix entries of $M$. If $g \in \PL_+(I^m,\partial I^m)$, let $\Sigma_i$ be simplices
of $I^m$ (of full dimension) such that the restriction of $g$ to each $\Sigma_i$ is affine,
and let $M_i(g)\in \GL(m,\R)$ be the linear part of $g$ on $\Sigma_i$. Then define
$$D(g)=\sup_i D(M_i(g))$$
\end{definition}

The next lemma relates the quantity $D$ to word length in finitely generated subgroups of
$\PL(I^m,\partial I^m)$.

\begin{lemma}\label{lemma:exponential_growth_of_D}
Let $G$ be a finitely generated subgroup of $\PL(I^m,\partial I^m)$, and let $S$ be a finite
generating set. There is an inequality
$$D(g) \le (m\cdot \max_{s\in S} D(s))^{|g|_S}$$
for any $g\in G$, where $|g|_S$ denotes word length of $g$ with respect to the given generating
set.
\end{lemma}
\begin{proof}
If $M$ and $N$ are elements of $\GL(m,\R)$, then $D(MN)\le mD(M)D(N)$. If $g,h \in \PL(I^m,\partial I^m)$
are arbitrary, with linear parts $M_i(g)$ and $N_j(h)$ with respect to subdivisions of $I^m$ into
simplices $\Delta_i$, $\Delta'_j$ then the linear part of $gh$ is everywhere equal to some 
$M_i(g)N_j(h)$; thus $D(gh) \le mD(g)D(h)$. From this the lemma follows.
\end{proof}

Thus the logarithm of $D$ is (up to an additive
constant) a subadditive function on $\PL(I^m,\partial I^m)$.
On the other hand, it turns out that under powers, $D(g^n)$ must grow at least {\em linearly} with
$n$:

\begin{proposition}\label{proposition:linear_D}
Let $g \in \PL(I^m,\partial I^m) - \Id$. Then there is an inequality
$$D(g^n) \ge nC$$
for some positive $C$ depending on $g$.
\end{proposition}
\begin{proof}
Let $\sigma$ be a simplex of codimension 1 in the frontier of $\fix(g)$. After subdividing $\sigma$
if necessary, we can find a simplex $\Delta$
on which $g$ is affine and nontrivial, in such a way that $\sigma$ is a face of $\Delta$. Then
for any point $p$ in the interior of $\sigma$, and for any $n$, the interior of $g^n(\Delta)\cap \Delta$ 
contains points arbitrarily close to $p$. If $M$ is the linear part of $g$ on $\Delta$, then
$g^n$ has linear part $M^n$ on some nonempty open subset. Since $M$ fixes a codimension 1 subset,
either $M$ has an eigenvalue of absolute value $\ne 1$ (in which case $D(g^n)$ grows exponentially)
or $M$ is a shear, in which case $D(g^n)$ grows linearly.
\end{proof}

From Lemma~\ref{lemma:exponential_growth_of_D} and Proposition~\ref{proposition:linear_D}
we immediately obtain the following corollary:

\begin{corollary}\label{corollary:distorted}
Every element of $\PL(I^m,\partial I^m)-\Id$ is at most exponentially distorted.
\end{corollary}

We can use this corollary to give many examples of locally indicable groups which do
not admit an injective homomorphism to any $\PL(I^m,\partial I^m)$.

\begin{example}
Let $G$ be the iterated HNN extension given by the presentation
$$G: = \langle a,b,c \; | \; bab^{-1} = a^2, cbc^{-1} = b^2 \rangle$$
Evidently, $b$ is exponentially distorted, and $a$ is super-exponentially distorted.
On the other hand, $G$ is locally indicable. To see this, first observe that there
is a surjection from $G$ onto the Baumslag-Solitar group 
$BS(1,2):=\langle b,c\; | \; cbc^{-1} = b^2\rangle$ sending $b \to b$, $c \to c$
and $a\to 1$. The group $BS(1,2)$ is
locally indicable, and the kernel is a free product of copies of
the locally indicable group $\Z[\frac 1 2]$, and therefore $G$ is
locally indicable as claimed.
\end{example}

It is interesting that the combinatorial structure of a PL map naturally gives rise to another
subadditive function on $\PL(I^m,\partial I^m)$ which can also be used to control distortion.

\begin{definition}[Triangle number]\label{definition:triangle_number}
Let $g \in \PL(I^m,\partial I^m)$. The {\em triangle number} of $g$, denoted $\Delta(g)$,
is the least number of $\PL$ simplices into which $I^m$ can be subdivided so that the
restriction of $g$ to each simplex is linear.
\end{definition}

Note that $\Delta(g^{-1}) = \Delta(g)$. For, if $\tau_g$ is a triangulation such that
$g$ is linear on each simplex, then the image $g(\tau_g)$ is a triangulation with the
same number of simplices, such that $g^{-1}$ is linear on each simplex.

The next lemma relates triangle number to word length in finitely generated subgroups of
$\PL(I^m,\partial I^m)$.

\begin{lemma}\label{lemma:exponentially_many_triangles}
Let $G$ be a finitely generated subgroup of $\PL(I^m,\partial I^m)$, and let $S$
be a finite generating set. There is a constant $C$ (depending only on the dimension $m$)
and an inequality
$$\Delta(g) \le (C\cdot \max_{s\in S} \Delta(s))^{|g|_S}$$
\end{lemma}
\begin{proof}
For each generator $s \in S$, let's let $\tau_s$ be a triangulation of $I^m$ by $\Delta(s)$
simplices such that the restriction of $s^{-1}$ to each simplex is linear. Now let $\sigma$
be any $\PL$ simplex in $I^m$. For each simplex $\tau$ of $\tau_s$ the intersection
$\tau \cap \sigma$ is convex and cut out by at most $2m+2$ hyperplanes (since each simplex
individually is cut out by $m+1$ hyperplanes). Thus it is one of finitely many combinatorial
types (depending only on the dimension $m$) and there is a constant $C$ such that this
intersection may be subdivided into at most $C$ linear simplices.

It follows that if $w \in G$ is arbitrary, and 
$\tau_w$ is a triangulation of $I^m$ by $\Delta(w)$ simplices such that the restriction of
$w$ to each simplex is linear, then by subdividing the intersections of simplices in
$\tau_s$ and $\tau_w$, we obtain a triangulation of $I^m$, which by construction can
be pulled back by $s^{-1}$ to a triangulation with at most
$C\Delta(s)\Delta(w)$ simplices, and
such that the product $ws$ is linear on each simplex.
The lemma follows by induction.
\end{proof}

The case of dimension 1 is much simpler. Triangle number counts simply the number of breakpoints
(plus 1), and can grow at most linearly:

\begin{lemma}\label{lemma:linearly_many_breakpoints}
Let $G$ be a finitely generated subgroup of $\PL(I,\partial I)$, and let $S$ be a finite
generating set. There is an inequality
$$\Delta(g)-1 \le \max_{s\in S} (\Delta(s)-1)\cdot |g|_S$$
\end{lemma}
\begin{proof}
Let $w \in G$ have breakpoints at $p_i$ and let $s$ have breakpoints at $q_j$. Then
each breakpoint of $ws$ is of the form $q_j$ or $s^{-1}(p_i)$.
\end{proof}

On the other hand, at least in dimension 2,
triangle number must grow {\em at least} linearly under powers:

\begin{proposition}\label{proposition:square_linear_triangle}
Let $g \in \PL(I^2,\partial I^2)-\Id$. Then there is an inequality
$$\Delta(g^n) \ge n$$
\end{proposition}
\begin{proof}
Consider the action of $g$ at a vertex in the frontier of $\fix(g)$. The projectivization
of this action is a piecewise projective action on the interval. We claim that there is
a vertex $v \in \fro(g)$ and a subinterval $J$ of the projective tangent bundle at $v$
which is fixed by $g$, where $g$ acts (topologically) conjugate to a translation, and
such that $g$ has a break point in the interior of $J$. Assuming this claim, the proof
proceeds as follows. Suppose (for the sake of argument) that $g$ moves
points on this interval in the positive direction, and suppose furthermore that there is
at least one breakpoint in this interval. Let $p$ be the uppermost
breakpoint. Denote by $[p,J^+]$ the positive subinterval of $J$ bounded (on one side)
by $p$. Then the points $g^i(p)$ for $0 < i < n$ are all distinct and contained in 
the open interval $(p,J^+)$ where the action of $g$ is smooth; thus by induction, $g^n$
is singular on $J$ at these points, so $\Delta(g^n)\ge n$. 

We now prove the claim. Let $\gamma$ be a component of the frontier of $\fix(g)$, so
that $\gamma$ is an immersed polyhedral circle. We further subdivide $\gamma$ if necessary, adding
vertices at points where $g$ is not locally linear. 
Denote the edges of $\gamma$ in cyclic order by $e_i$ (indices taken cyclically), and orient
these edges so they point in the positive direction around $\gamma$ (with respect to the
orientation it inherits as a boundary component of the region $P$ that $g$ does not fix pointwise).
Let $v_i$ be the initial vertex of $e_i$. Subdivide $P$ near $v_i$ into polyhedra on which
the element $g$ acts linearly, and let $e_i^+$ and $e_i^-$ be the extremal edges of this
subdivision which point into the interior of $P$, where the choice of $\pm$ is such 
that the tangents to the edges 
$e_i, e_i^-, e_i^+$ at $v_i$ are positively circularly ordered in the unit tangent circle at 
$v_i$. We think of $\gamma$ now as a map $\gamma:S^1 \to I^2$ and pull back the unit tangent
bundle over the image to a circle bundle $U$ over $\gamma$. Let $V\subset U$ be the subset
of tangents pointing locally into the interior of $P$; thus $V$ is an (open) interval bundle
over $S^1$. The argument function lifts locally (in $\gamma$) to a real valued function on
each fiber, in such a way that holonomy around $\gamma$ in the positive direction
increases the argument by $2\pi$
(this is just a restatement of the fact that the winding number of an embedded loop is 1
around the region it bounds). Since
the (real-valued) argument of $e_i^-$ in each fiber is at most equal to the argument of
$e_i^+$, it follows that there must be some index $i$ for which the argument of $e_{i+1}^+$
is {\em strictly} bigger than that of $e_i^-$, as measured in the bundle $V$ 
restricted to the contractible
subset $e_i \subset \gamma$ (where the argument is globally well-defined). 
If $P_i$ is the region of $P$ on which $g$ is linear which
is bounded on three (consecutive) sides by $e_i^-$, $e_i$, $e_{i+1}^+$, then  
$e_i^-$ and $e_{i+1}^+$ can be extended to bound a triangle $T$ in such a way that $T \cap P_i$ contains a 
neighborhood of $e_i$ in $I^2 - \fix(g)$. See Figure~\ref{frontier_triangle}.

\begin{figure}[htpb]
\labellist
\small\hair 2pt
\pinlabel $e_i$ at 130 176
\pinlabel $e_i^-$ at 170 170
\pinlabel $e_{i+1}^+$ at 95 163
\pinlabel $\gamma$ at 83 88
\pinlabel $T$ at 135 140
\endlabellist
\centering
\includegraphics[scale=0.75]{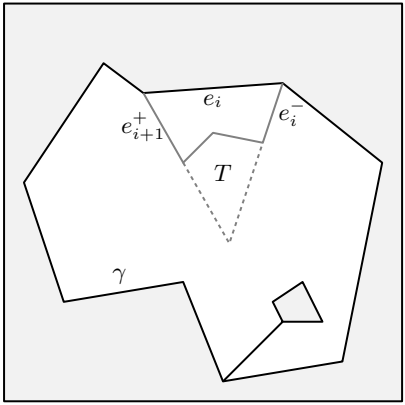}
\caption{The (dashed) triangle $T$ associated to the region $P_i$}\label{frontier_triangle}
\end{figure}

A linear map of the plane
which fixes the edges of a triangle is trivial; but since the interior of $P_i$ is not
in $\fix(g)$, it follows that $g$ must move at least one of the edges 
$e_i^-$ and $e_{i+1}^+$ off itself;
the projective action of $g$ at the corresponding vertex must therefore move some 
breakpoint off itself, proving the claim, and therefore the proposition.
\end{proof}

The argument of Proposition~\ref{proposition:square_linear_triangle} actually shows $\Delta(g^n)\ge n$
for any $g \in \PL(I,\partial I)$, since if $p$ is an uppermost breakpoint of $g$, the points
$g^i(p)$ are all breakpoints of $g^n$ for $0<i<n$. Thus by Lemma~\ref{lemma:linearly_many_breakpoints}
we conclude that every nontrivial element is undistorted in $\PL(I,\partial I)$.

\begin{question}
Is there a nontrivial distorted element in $\PL(I^2,\partial I^2)$? How about 
$\PL(I^m,\partial I^m)$ for any $m\ge 2$?
\end{question}

\begin{question}
Does the argument of Proposition~\ref{proposition:square_linear_triangle} generalize to
dimensions $n\ge 3$?
\end{question}

\section{\texorpdfstring{$C^1$}{C1} group actions}\label{section:smooth}

In this section, for the sake of completeness and to demonstrate the similarity and difference
of methods, we prove the analog of Theorem~\ref{theorem:pillow}
for groups acting smoothly on manifolds, fixing a codimension 1 submanifold pointwise.

\subsection{Definitions}

\begin{notation}
Let $M$ be a smooth manifold (possibly with smooth boundary) and let $K$ be
a closed smooth submanifold of $M$. We denote by $\Diff^1(M,K)$ the group of
$C^1$ self-diffeomorphisms of $M$ which are fixed pointwise on $K$. If $M$ is
orientable, we denote by $\Diff^1_+(M,K)$ the subgroup of orientation-preserving
$C^1$ self-diffeomorphisms.
\end{notation}

The main example of interest is the smooth closed unit ball $D^n$ in $\R^n$,
and its boundary $\partial D^n = S^{n-1}$.

One of the most important theorems about groups of $C^1$ diffeomorphisms is the
{\em Thurston Stability Theorem}:

\begin{theorem}[Thurston Stability Theorem, \cite{Thurston_stability}]\label{theorem:Thurston_stability_theorem}
Let $G$ be a group of germs of $C^1$ diffeomorphisms of $\R^n$ fixing $0$.
Let $\rho:G \to \GL(n,\R)$ be the {\em linear part of $G$ at $0$}; i.e.\/ the
representation sending $g \to dg|_0$, and let $K$ be the kernel of $\rho$. Then
$K$ is locally indicable.
\end{theorem}

\subsection{Left orderability}\label{C1_LO_subsection}

The purpose of this section is to prove the following {\em $C^1$ LO Theorem}:

\begin{theorem}[$C^1$ LO]\label{theorem:cilo}
Let $M$ be an $n$ dimensional connected smooth manifold, and let $K$ be a
nonempty $n-1$ dimensional closed submanifold. Then the group $\Diff_+^1(M,K)$ is
left orderable; in fact it is locally indicable.
\end{theorem}

An important special case is $M=D^n$ and $K=\partial D^n = S^{n-1}$.

\begin{definition}[Fixed rank]\label{definition:fixed_rank}
Let $f\in \Diff^1(M,K)$. The {\em fixed rank} of $f$ at a
point $p$ is defined to be $-1$ if $f$ does not fix $p$, and otherwise to be
the dimension of the subspace of $T_p$ fixed by $df|_p$. Similarly, if $G$
is a subgroup of $\Diff^1(M,K)$, the
{\em fixed rank} of $G$ at a point $p$ is $-1$ if $p$ is not in $\fix(G)$, and
otherwise is the dimension of the subspace of $T_p$ fixed by 
$\rho_p:G \to \GL(n,\R)$ sending $g \to dg|_p$ for all $g\in G$.
\end{definition}

Recall that a (real-valued) function $f$ on a topological space is 
{\em upper semi-continuous} if for every point $p$ and every $\epsilon>0$
there is a neighborhood $U$ of $p$ so that $f(x)\le f(p)+\epsilon$ for
all $x\in U$. In other words, the value of $f$ can ``jump up'' at a limit,
but not down. If $f$ is upper semi-continuous, then for any $t$ the
set where $f \ge t$ is closed.

\begin{lemma}\label{lemma:fixed_rank_semicontinuous}
Let $G$ be a subgroup of $\Diff^1(M,K)$.
Then the fixed rank of $G$ is upper semi-continuous on $M$. Consequently
the subset where the fixed rank is $\ge m$ is {\em closed}, for any $m$.
\end{lemma}
\begin{proof}
Let $p_i \to p$ and suppose that $G$ has fixed rank $\ge m$ on all $p_i$.
If $m=-1$ there is nothing to prove. Otherwise, the $p_i$ are all fixed
by every $g\in G$ and there is a subspace $\pi_i \subset T_{p_i}$ of
dimension $m$ fixed by every $dg|_{p_i}$. By compactness of the Grassmannian 
of $m$-dimensional subspaces of $\R^n$, after passing to a subsequence we
can assume that $\pi_i$ converges to some $m$-dimensional subspace
$\pi \subset T_p$. Observe that for every $g$ we must have that $dg|_p$
fixes $\pi$, or else some $dg|_{p_i}$ would fail to fix $\pi_i$, since the
action is $C^1$. Thus the fixed rank of $G$ at $p$ is $\ge m$.
\end{proof}

We can now give the proof of Theorem~\ref{theorem:cilo}:

\begin{proof}
Let $G$ be a nontrivial 
finitely generated subgroup of $\Diff_+^1(M,K)$, and
let $X$ be the subset of $G$ where the fixed rank is $\ge n-1$. Since
$G$ is arbitrary, it suffices
to show that $G$ admits a surjection to $\Z$.

By Lemma~\ref{lemma:fixed_rank_semicontinuous} the set $X$ is closed. Furthermore,
it is nonempty, since it includes $K$. Since $G$ is nontrivial,
some point of $M$ is moved by some element of $G$, and therefore the fixed rank
is $-1$ at that point; since $M$ is connected, it follows that 
the frontier $\fro(X)$ is nonempty.

Let $p \in \fro(X)$ be arbitrary, and let $\pi \subset T_pM$ be an
$n-1$ dimensional plane fixed by $dg|_p$ for all $g\in G$.
Let $\rho:G \to \GL(n,\R)$ send $g \to dg|_p$. By Lemma~\ref{lemma:linear_stabilizer}
and Example~\ref{example:codimension_1}
the image of $\rho$ is locally indicable. If the image is nontrivial we are
done. So suppose the image is trivial. Since $p$ is in the frontier of
$X$ it must be in the frontier of $\fix(G)$ (or else it would be in the interior
of the set where the fixed rank is $n$) so the image of $G$ in the group of
germs of diffeomorphisms fixing $p$ is nontrivial. So the Thurston stability
theorem (i.e.\/ Theorem~\ref{theorem:Thurston_stability_theorem}) implies
that $G$ surjects to $\Z$. This completes the proof.
\end{proof}

\section{Bi-orderability}\label{section:biorderable}

If a left-order $\prec$ of a group $G$ is also invariant under right-multiplication, 
we'll call it a {\em bi-order} and say that $G$ is {\em bi-orderable}.  It is easy 
to see that a left-order $\prec$ is a bi-order if and only if its positive cone 
$$P := \{ g \in G \text{ such that } 1 \prec g\}$$ 
is invariant under conjugation; i.e.\/ if and only if
$g^{-1}Pg \subset P$ for all $g \in G$.

In this section we discuss bi-orderability for various subgroups of
$\Homeo(I^n,\partial I^n)$ in specific dimensions, and with various kinds of regularity.

\subsection{Bi-orderability in general}

First of all, it should be noted that the distinction between left orderability and bi-orderability
is not vacuous:

\begin{example}\label{example:product_of_conjugates}
If $G$ is a group in which there is a nontrivial element $g$ so that some product of
conjugates of $g$ is trivial, then $G$ is not bi-orderable. For, in any left ordering,
we may assume that $g \in P$ (up to reversing the order). So, for example,
the Klein bottle group $\langle a,b \; |\; aba^{-1} = b^{-1} \rangle$ is not bi-orderable,
though it {\em is} locally indicable.
\end{example}

One may summarize the existence of LO groups which are not bi-orderable in the following
proposition:

\begin{proposition}
$\Homeo(I, \partial I)$ is not bi-orderable.
\end{proposition}

This is because $\Homeo(I, \partial I)$ contains isomorphic copies of all 
countable LO groups, many of which are not bi-orderable.  Suspending to 
higher dimensions, we conclude:

\begin{corollary}
For all $n \ge 1$, $\Homeo(I^n, \partial I^n)$ is not bi-orderable.
\end{corollary}

\subsection{Bi-orderability in PL}

The following was observed in a slightly different context by Chehata \cite{Chehata}.

\begin{proposition}
$\PL(I, \partial I)$ is bi-orderable.
\end{proposition}

One can take as positive cone $P$ for a bi-ordering all PL functions $f:I \to I$ whose graph departs from the diagonal for the first time with slope greater than 1.  In other words, if $x_0$ is the maximal element of $I = [-1, 1]$ such that $f(x) = x$ for all $x \in [-1, x_0]$, then for all sufficiently small 
$\epsilon > 0$ we have $f(x_0 + \epsilon) > x_0 + \epsilon$.  Then $P$ is closed under composition of functions, conjugation, and $\PL(I, \partial I) = \{1\} \sqcup P \sqcup P^{-1}$, so it defines a bi-order by the recipe $g \prec h \iff g^{-1}h \in P$.

\begin{proposition}\label{proposition:non-biorderable}
$\PL_\omega(I^2, \partial I^2)$ is not bi-orderable.
\end{proposition}

To see this, we consider two functions $f, g \in \PL_\omega(I^2, \partial I^2)$ defined 
as follows.  Let $h$ denote the function illustrated in Figure~\ref{elementary_twist}.  
Recall that $h^{12}$ is a Dehn twist, which is the identity on the inner square, as well 
as on $\partial I^2$.  Let $f := h^6$, so that $f$ rotates the inner square by 180 
degrees.  Referring to Figure~\ref{construction}, define $g$ to be the identity 
outside the little squares, which are strictly inside the inner square rotated by $f$.  
On the little square on the left, let $g$ act as $h$, suitably scaled, and on the 
little square to the right let $g$ act as $h^{-1}$.  Noting that $f$ interchanges 
the little squares, and that $h$ commutes with 180 degree rotation, one checks 
that $f^{-1}gf = g^{-1}$.

Such an equation cannot hold (for $g$ not the identity) in a bi-ordered group, as 
it would imply the contradiction that $g$ is greater than the identity if and only 
if $g^{-1}$ is greater than the identity (this is just the Klein bottle group
from Example~\ref{example:product_of_conjugates}).

\begin{figure}[htpb]
\labellist
\small\hair 2pt
\pinlabel $h$ at 100 125
\pinlabel $h^{-1}$ at 150 125
\pinlabel $\Id$ at 125 175
\endlabellist
\centering
\includegraphics[scale=0.6]{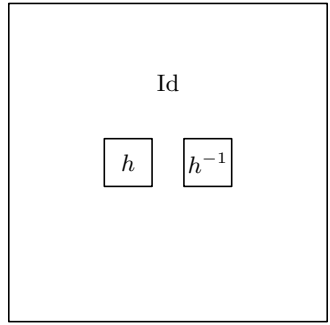}
\caption{Building the function $g \in \PL_\omega(I^2, \partial I^2)$}\label{construction}
\end{figure}

\begin{corollary}
For $n \ge 2$ none of the groups $\PL_\omega(I^n, \partial I^n), \PL(I^n, \partial I^n)$ 
is bi-orderable.
\end{corollary}

This follows from Lemma \ref{lemma:suspension_of_PL}, noting that suspension also 
takes $\PL_\omega(I^n, \partial I^n)$ isomorphically into 
$\PL_\omega(I^{n+1}, \partial I^{n+1})$.

Recall Theorem \ref{theorem:pillow}: if $M$ is a connected PL $n$-manifold and $K$ 
is a nonempty closed PL $n-1$ dimensional submanifold, then $\PL(M,K)$ is LO.  
Noting that one may include $\PL(I^n, \partial I^n)$ in $\PL(M,K)$, by acting 
on a small $n$-ball in $M$, we see also

\begin{corollary}
For $n \ge 2$, $\PL(M,K)$ is not bi-orderable.
\end{corollary}

\subsection{Bi-orderability in Diff}

By an argument similar to Proposition \ref{proposition:non-biorderable} one can show 
\begin{proposition}
For $n \ge 2$ and for any $0 \le p\le \infty$, the group $\Diff^p(I^n, \partial I^n)$ is not bi-orderable.
\end{proposition}

Note that in every case bi-orderability is ruled out by the existence of a
nontrivial element which is conjugate to its inverse.

For $n=1$ and $p=0,1$ bi-orderability in $\Diff^p(I,\partial I)$ is ruled out for
the same algebraic reason:

\begin{example}[$C^1$ example]
For $i \in \Z$ let $x_i$ be a discrete, ordered subset of the interior of $I$
accumulating at the endpoints like the harmonic series. Let $f$ be a $C^1$
diffeomorphism whose fixed point set in the interior of $I$
is exactly the union of the $x_i$, and which alternates between translating
in the positive and negative direction on intervals $(x_i,x_{i+1})$ for $i$
respectively even and odd. Let $g$ take $x_i$ to $x_{i+1}$ and conjugate the
action of $f$ to $f^{-1}$. This construction can evidently be made $C^1$ in
the interior. Moreover, if we arrange for $f'$ to converge uniformly to $0$ towards
the endpoints, then the same is true of $g$, since $g$ acts there almost like
a linear map $[1-1/i,1] \to [1-1/(i+1),1]$, whose derivatives converge uniformly
to $0$.
\end{example}

On the other hand, this construction cannot be made $C^2$, by Kopell's Lemma,
which says the following:

\begin{theorem}[Kopell's Lemma \cite{Kopell}, Lemma~1]\label{theorem:Kopell_lemma}
Let $f$ and $g$ be commuting elements of $\Diff^2(I,\partial I)$. If $g$ has
no fixed point in the interior of $I$ but $f$ has a fixed point in the interior of
$I$ then $f=\Id$.
\end{theorem}

\begin{corollary}
In $\Diff^2(I,\partial I)$ no nontrivial element is conjugate to its inverse.
\end{corollary}
\begin{proof}
If $f$ is nontrivial and conjugate to its inverse, there are some subintervals where
it acts as a positive translation, and some where it acts as a negative translation,
and some point between the two is fixed. Let $g$ conjugate $f$ to $f^{-1}$.
Applying Kopell's lemma to $g$ and $f^2$ we see that $f^2=\Id$ so $f=\Id$.
\end{proof}

In particular, the Klein bottle group does not embed in $\Diff^2(I,\partial I)$.

\begin{remark}\label{remark:biorderable}
We suspect that the group $\Diff^p(I,\partial I)$ is
not bi-orderable for any finite $p$, but have not been able to show this. 
In this remark we give a detailed construction of a $C^\infty$ group action containing two
specific elements which are $C^1$ conjugate; if the construction could be done in such a way that the
elements are $C^p$ conjugate, it would prove that $\Diff^p(I,\partial I)$ is not bi-orderable,
but we have not been able to show this.

We build a $C^p$ diffeomorphism $f$ of $I$ as follows (for simplicity of formulae, we take
here $I$ to be the interval $[0,1]$ instead of $[-1,1]$ as throughout the rest of the paper). 
For $i<0$ let $p_i = 2^i$ and
for $i\ge 0$ let $p_i = 1-2^{-i-2}$. We build $f$ fixing each $p_i$, and acting alternately
as a positive and negative translation on complementary regions. We will insist that $f$ is
infinitely tangent to the identity at each $p_i$. For concreteness, on each
of the intervals $[7/8,15/16]$ and $[1/16,1/8]$ we let $f$ be the time $1$ flow of a
$C^\infty$ vector field with no zeros in the interior of the intervals, and infinitely tangent
to zero at the endpoints. The definition of $f$ on the rest of $I$ will be defined implicitly
in what follows.

Now, we let $g$ be a diffeomorphism which takes $p_i$ to $p_{i+1}$
for $i\ge 0$ and takes $p_i$ to $p_{i-1}$ for $i<0$, and let $g$ act as a dilation
(i.e.\/ with locally constant nonzero derivative) on open
neighborhoods of $0$ and $1$. Evidently we can choose such a $g$ to be $C^{\infty}$.
Finally, we arrange the dynamics of $f$ on the complementary intervals
so that $g$ conjugates $f$ to $f^{-k}$ outside the interval $[1/4,7/8]$ for some big $k$,
depending on $p$ (which we will specify shortly). Since $f$ (on the intervals $[7/8,15/16]$
and $[1/16,1/8]$) is the time 1 flow of a vector field, taking an $n$th root multiplies each
$C^p$ norm by $O(1/n)$ when $n$ is big. On the other hand, conjugating a diffeomorphism by
dilation by $1/\lambda$ blows up the $C^p$ norm by a factor of $\lambda^{p-1}$,
so providing $k>2^{p-1}$ the diffeomorphism $f$ we build with this property will be
$C^\infty$ on $(0,1)$, and $C^p$ tangent to the identity at $0$ and $1$, and therefore is an element of
$\Diff^p(I,\partial I)$. By construction, the product $h:=f^kgfg^{-1}$ is supported on
the interval $[1/4,7/8]$ which is subdivided into 3 intervals $J_0,J_1,J_2$ where
$h$ acts alternately as positive, negative, positive translation. If we are careful in our
choice of $g$, we may ensure that $h$ is infinitely tangent to the identity at the endpoints
of the $J_i$.

By a similar 
construction, we can choose a slightly different element $j$, whose germs at $0$ and $1$
agree with $g$, so that $j$ conjugates $f$ to $f^{-k}$ outside the interval $[1/8,3/4]$,
and then we can arrange that $h':=f^kjfj^{-1}$ is supported on the interval $[1/8,3/4]$
and acts alternately on a subdivision $J_0',J_1',J_2'$ as negative, positive, negative
translation, and infinitely tangent to the identity at the endpoints.

Now, suppose for a suitable choice of $g$ and $j$ as above, we could find some
$e$ which conjugates $h'$ to $h^{-1}$, so that we get
a relation of the form
$$e(f^kjfj^{-1})e^{-1}f^kgfg^{-1} = \Id$$
But this word is a product of conjugates of $f$, and therefore a relation of this
kind would certify that $\Diff^p(I,\partial I)$ is not 
bi-orderable. We suspect that such an $e$ can be found (for some $g$ and $j$), but leave this
as an open problem.
\end{remark}

In view of the construction outlined in
Remark~\ref{remark:biorderable} it is a bit surprising that the
group $\Diff^\omega(I,\partial I)$ {\em is} bi-orderable:

\begin{proposition}\label{prop:C_omega_biorderable}
The group $\Diff^\omega(I,\partial I)$ is bi-orderable.
\end{proposition}
\begin{proof}
Define a bi-ordering on $\Diff^\omega(I,\partial I)$ as follows. If the derivative $f'(0)$
is $>1$, put the element $f$ in the positive cone. If $f'(0)=1$,
look at the Taylor expansion at $0$. This is of the form
$$f(t) = t + a_2t^2 + a_3t^3 + \cdots$$
Since $f$ is real-analytic, either $f=\Id$, or else there is some first index $i$
such that $a_i \ne 0$. Then put $f$ in the positive cone if $a_i>0$.

One checks that this really does define a cone, and that for every nontrivial element
$f$, either $f$ is in the cone, or $f^{-1}$ is in the cone. Finally, the cone is
conjugation invariant.

More geometrically, $f$ is in the positive cone if $f(t)>t$ for all {\em sufficiently small}
$t>0$. That this is conjugation invariant and a cone is obvious. That it is well-defined
follows from the fact that a nontrivial real analytic diffeomorphism of $I$ can't have
infinitely many fixed points.
\end{proof}

\section{Groups fixing codimension 2 submanifolds}\label{section:circular}

In this section we establish analogs of Theorem~\ref{theorem:pillow} and
Theorem~\ref{theorem:cilo} for groups acting in a PL or smooth manner on
manifolds, fixing a submanifold of codimension 2 pointwise. The conclusion is
not that such groups are left orderable (they are not in general) but that they
are {\em circularly orderable}. 

\subsection{Circularly ordered groups}

We recall standard definitions and facts about circular orders; \cite{Calegari_foliations},
Chapter~2 is a reference.

\begin{definition}\label{definition:circular_order}
Let $\Sigma$ be a set and let $\Sigma^{(3)}$ denote the space of {\em distinct}.
ordered triples in $\Sigma$. A {\em circular order} on $\Sigma$ is a map 
$e:\Sigma^{(3)} \to \pm 1$ satisfying the {\em cocycle condition}
$$e(s_1,s_2,s_3) - e(s_0,s_2,s_3) + e(s_0,s_1,s_3) - e(s_0,s_1,s_2) = 0$$
for all distinct quadruples $s_0,s_1,s_2,s_3$ in $\Sigma$.

A (left) {\em circular order} on a group $G$ is a circular order on $G$ which
is invariant under left-multiplication; i.e.\/ it satisfies
$$e(gg_1,gg_2,gg_3) = e(g_1,g_2,g_3)$$
for all $g$ in $G$ and all distinct triples $(g_1,g_2,g_3) \in G^{(3)}$.
\end{definition}

\begin{remark}
Note that $e$ may be extended to all triples $\Sigma^3$ by defining $e$ to be
$0$ on $\Sigma^3 - \Sigma^{(3)}$; this extension also satisfies the cocycle
condition.
\end{remark}

\begin{notation}
A group which admits an invariant circular order is said to be 
{\em circularly orderable}, which we usually abbreviate by CO.
\end{notation}

\begin{lemma}\label{lemma:local_CO_is_CO}
Let $G$ be a group. Suppose every finitely generated subgroup of $G$ is CO.
Then $G$ is CO.
\end{lemma}

\begin{lemma}\label{lemma:extension_is_CO}
If there is a short exact sequence $0 \to K \to G \to H \to 0$ where $H$ is
circularly orderable and $K$ is left orderable, then $G$ is circularly orderable.
\end{lemma}

In particular, left orderable groups are circularly orderable.

\begin{lemma}\label{lemma:action_is_CO}
A countable group is circularly orderable if and only if it is isomorphic to
a subgroup of $\Homeo_+(S^1)$. Conversely, {\em every} subgroup of $\Homeo_+(S^1)$
is circularly orderable.
\end{lemma}

\begin{lemma}\label{lemma:euler_is_obstruction}
Suppose $G$ is circularly orderable. The (extended) function $e$ is a $\Z$-valued
2-cocycle on $G$ and thereby determines a class $[e] \in H^2(G;\Z)$. If the
class $[e]=0$ then $G$ is left orderable.
\end{lemma}

\begin{example}\label{example:GL_2_CO}
The group $\GL^+(2,\R)$ is an extension $0 \to \R^+ \to \GL^+(2,\R) \to \SL(2,\R) \to 0$.
The group $\SL(2,\R)$ acts faithfully and in an orientation-preserving way
on the circle of linear rays through the origin, and is therefore circularly
orderable, by Lemma~\ref{lemma:action_is_CO}. By Lemma~\ref{lemma:extension_is_CO}
the group $\GL^+(2,\R)$ is CO.
\end{example}

\subsection{Knots}

We would like to extend the results from \S~\ref{PL_LO_subsection} 
and \S~\ref{C1_LO_subsection} to pairs $(M,K)$ where the codimension 
of $K$ is 2. Motivated by the interesting example where $M=S^3$ and $K$ is a circle, 
we use the informal term ``knot'' for $K$ although we do not assume either that
$K$ is connected, or that it is homeomorphic to a sphere.

\begin{theorem}[$C^1$ Knots CO]\label{theorem:ci_ropes_co}
Let $M$ be an $n$ dimensional connected smooth orientable manifold, and let $K$ be a
nonempty $n-2$ dimensional closed submanifold. Then the group $\Diff^1_+(M,K)$
is circularly orderable.
\end{theorem}
\begin{proof}
Fix some point $k \in K$. The linear representation $\rho:\Diff^1_+(M,K) \to \GL(n,\R)$
defined by $\rho(g) = dg|_k$ fixes a subspace $\pi$ of dimension $n-2$, and
therefore by Lemma~\ref{lemma:linear_stabilizer} 
has image in $\R^{2(n-2)}\rtimes \GL^+(2,\R)$. The image group
is CO by Lemma~\ref{lemma:extension_is_CO}, since
$\GL^+(2,\R)$ is CO (by Example~\ref{example:GL_2_CO}) and $\R^{2(n-2)}$
is locally indicable and therefore LO. Let $K$ denote the kernel, and let
$G$ be a finitely generated subgroup of the kernel. If the germ of $G$ at $k$
is nontrivial, then $G$ surjects onto $\Z$, by the Thurston stability theorem
(i.e.\/ Theorem~\ref{theorem:Thurston_stability_theorem}).
If the germ of $G$ at $k$ is trivial, then each of the finite generators $g_i$
fixes an open neighborhood of $k$, and therefore $\fix(G)$ has nonempty interior.
It follows that the subset $X$ where the fixed rank of $G$ is $\ge n-1$ is
nonempty, and since $M$ is connected $\fro(X)$ is also nonempty.
Therefore as in the proof of Theorem~\ref{theorem:cilo} we deduce that $G$ 
surjects onto $\Z$. It follows that $K$ is locally indicable and therefore LO,
so $\Diff^1_+(M,K)$ is CO by Lemma~\ref{lemma:extension_is_CO}.
\end{proof}

Before proving the corresponding theorem for $\PL$ actions, we must analyze the
local structure of a group of $\PL$ homeomorphisms at an arbitrary point on an
$m$-dimensional fixed submanifold.

\begin{lemma}\label{lemma:PL_fixed_germ}
Let $G$ be a countable group of germs at $0$ of $\PL$ diffeomorphisms of $\R^n$ fixing
$\R^m$. There is a homomorphism from $G$ to the group of piecewise projective
automorphisms of $\RP^{n-m-1}$, and for every finitely generated subgroup $H$
of the kernel, either $H$ surjects to $\Z$ or $\fix(H)$ 
has nonempty interior arbitrarily close to $0$. In particular, the kernel is
locally indicable, and therefore LO.
\end{lemma}
\begin{proof}
A linear automorphism of $\R^n$ fixing $\R^m$ acts linearly on the quotient
$\R^n/\R^m$ and therefore projectively on $\RP^{n-m-1}$. For each element
$g \in G$ we can pick a subdivision into linear simplices such that the restriction
of $g$ to each simplex is linear. The set of points on $\R^m$ in the complement
of the $m-1$ skeleton of this subdivision is open and dense, and therefore since
$G$ is countable, there is a point $p$ on $\R^m$ near $0$ which is in the complement
of the $m-1$ skeleton of the subdivision associated to {\em every} $g \in G$.

For every $g \in G$ and every simplex $\sigma$ in the linear subdivision 
associated to $g$, there is a well-defined projective action of $g$ on each
simplex of the link of $g$. Putting these actions together for all $g \in G$
gives a piecewise projective action of $G$ on $\RP^{n-m-1}$ at $p$. 

A finitely generated subgroup $H$ of the kernel locally preserves the foliation by 
planes parallel to $\R^m$, and acts there by translation. If this action is
trivial, $\fix(H)$ has interior near $p$.
\end{proof}

\begin{theorem}[$\PL$ Knots CO]\label{theorem:PL_ropes_co}
Let $M$ be an $n$ dimensional connected $\PL$ orientable manifold, and let $K$ be a
nonempty $n-2$ dimensional closed submanifold. Then the group $\PL_+(M,K)$ is
circularly orderable.
\end{theorem}
\begin{proof}
Let $G$ be a countable subgroup of $\PL_+(M,K)$ and consider the action near
a point $p$ in an $n-2$ dimensional simplex in $K$. By Lemma~\ref{lemma:PL_fixed_germ}
there is a homomorphism from the germ of $G$ at $p$ to the group of piecewise projective
automorphisms of $\RP^1$ with LO kernel. Thus the germ of $G$ at $p$ is CO.
Any finitely generated subgroup $H$ of the kernel of the map from $G$ to the germ
at $p$ has the property that $\fix(H)$ has nonempty interior, and therefore
$H$ is LO (in fact, locally indicable) as in the proof of Theorem~\ref{theorem:pillow}.
Thus every countable subgroup of $\PL_+(M,K)$ is CO, and therefore $\PL_+(M,K)$ is
CO by Lemma~\ref{lemma:local_CO_is_CO}.
\end{proof}

\subsection{Hyperbolic knots}

An interesting application is to groups acting on a 3-manifold stabilizing a knot.
Explicitly, let's specialize to the case that $M$ is an orientable 3-manifold, 
and $K$ is a knot with hyperbolic complement.
Let $G(M,K)$ denote the group of orientation preserving 
$\PL$ (resp. $\Diff^1$) homeomorphisms of $M$ that take $K$ to itself
by an orientation preserving homeomorphism, where we do {\em not} assume $K$ is
fixed pointwise; and let $G_0(M,K)$ denote the subgroup of such transformations
isotopic to the identity.

Since $M-K$ is Haken, as a topological group, 
the homotopy type of $G(M,K)$ is well-understood by the work of Hatcher and Ivanov. In fact,
one has:

\begin{theorem}[Hatcher; Hatcher, Ivanov]
As a topological group, $G_0(M,K)$ is contractible.
\end{theorem}

The PL case is due independently to Hatcher and Ivanov; see e.g.\/ \cite{Hatcher} or \cite{Ivanov}.
As was well-known at the time, the smooth case follows from the PL case once one
knows the Smale conjecture, proved by Hatcher \cite{Hatcher_Smale}.

\begin{corollary}
With notation as above, $G_0(M,K)$ is left orderable.
\end{corollary}
\begin{proof}
Let's abbreviate $G_0(M,K)$ by $G$ in what follows.
The image of $G$ in $\Homeo_+(K)$ is circularly orderable, and we have
already shown that the kernel is circularly orderable. We would like to show that
both the image and the kernel are actually LO. By Lemma~\ref{lemma:euler_is_obstruction} we need
to show that the cohomology class $[e]$ associated to either circular order 
is trivial. Notice by construction that $[e]$ is an element of
$H^2(G;\Z)$ where $G$ is thought of as a {\em topological} group. But since
$M-K$ is hyperbolic, $G$ is contractible, and therefore $H^2(G;\Z)=0$.
\end{proof}

\section{Groups of homeomorphisms}\label{section:finite}

It is natural to wonder whether the groups $\Homeo(M,K)$ are left- or circularly-ordered
when $K$ is of codimension 1 or 2, by analogy with the PL or smooth case. But here the 
situation is utterly mysterious, even in dimension 2. In fact, the following question still
seems far beyond reach:

\begin{question}
Is $\Homeo(I^2,\partial I^2)$ left-orderable?
\end{question}

It is challenging to obtain any restrictions on the subgroups of $\Homeo(I^n,\partial I^n)$ at
all, for $n\ge 2$. The purpose of this section is to point out that Smith theory shows that
the groups $\Homeo(I^n,\partial I^n)$ are torsion free. This fact seems to be well known to
the experts in Smith theory and its generalizations, but not to people working on left-orderable
groups, and therefore it seems worthwhile to give a proof (or rather to point out how the result
follows immediately from results which are well-documented in the literature).

We use the following theorem, generalizing work of Smith, and due in this generality to Borel \cite{Borel}:

\begin{theorem}[Smith, Borel]\label{theorem:Smith}
Let $X$ be a {\em finitisic space} (e.g. a compact space) with the mod $p$ homology of a point,
and let $G$ be a finite $p$-group acting on $X$ by homeomorphisms. Then $X^G$ (the fixed point set of
$G$) has the mod $p$ homology of a point.
\end{theorem}

Here a {\em finitistic space} is one such that every covering has a finite dimensional refinement.
Any compact space is finitistic, so $I^n$ is certainly finitistic (and the main application of
Smith theory is to manifolds and manifold-like spaces).

\begin{corollary}
The group $\Homeo(I^n,\partial I^n)$ is torsion-free for all $n$.
\end{corollary}
\begin{proof}
If $f$ is a nontrivial periodic homeomorphism of $I^n$ fixed on $\partial I^n$, then some power of
$f$ has prime order $p$ for some nontrivial $p$, so without loss of generality we may assume that
$f$ itself has order $p$. By Theorem~\ref{theorem:Smith} the fixed point
set of $f$ has the mod $p$ homology of a point. But this fixed point set includes $\partial I^n$,
which is homologically essential (with mod $p$ coefficients) in the complement of any point in $I^n$.
Thus all of $I^n$ is fixed by $f$.
\end{proof}


\begin{thebibliography}{99}
\bibitem{Agol}
	I. Agol,
	\emph{The virtual Haken conjecture},
	preprint, arXiv:1204.2810
\bibitem{Borel}
	A. Borel,
	\emph{Nouvelle d\'emonstration d'un th\'eor\`eme de P. A. Smith},
	Comment. Math. Helv. {\bf 29} (1955), 27--39
\bibitem{Brin_Squier}
	M. Brin and C. Squier,
	\emph{Groups of piecewise linear homeomorphisms of the real line},
	Invent. Math. {\bf 79} (1985), no. 3, 485--498
\bibitem{Burns_Hale}
	R. Burns and V. Hale,
	\emph{A note on group rings of certain torsion-free groups},
	Canad. Math. Bull. {\bf 15} (1972), 441--445
\bibitem{Calegari_foliations}
	D. Calegari,
	\emph{Foliations and the geometry of 3-manifolds},
	Oxford Mathematical Monograps. Oxford University Press, Oxford, 2007
\bibitem{Calegari_PL}
	D. Calegari,
	\emph{Stable commutator length in subgroups of $\PL^+(I)$},
	Pacific J. Math. {\bf 232} (2007), no. 2, 257--262
\bibitem{Calegari_scl}
	D. Calegari,
	\emph{scl},
	MSJ Memoirs {\bf 20}, Mathematical Society of Japan, Tokyo, 2009
\bibitem{Calegari_Freedman}
	D. Calegari and M. Freedman,
	\emph{Distortion in transformation groups},
	Geom. Topol. {\bf 10} (2006), 267--293
\bibitem{Chehata}
	C. Chehata,
	\emph{An algebraically simple ordered group},
	Proc. LMS {\bf 3} (1952), no. 2, 183--197
\bibitem{Franks_Handel}
	J. Franks and M. Handel,
	\emph{Distortion elements in group actions on surfaces},
	Duke Math. J. {\bf 131} (2006), no. 3, 441--468
\bibitem{Funar}
	L. Funar,
	\emph{On power subgroups of mapping class groups},
	preprint, arXiv:0910.1493
\bibitem{Gambaudo_Ghys}
	J.-M. Gambaudo and \'E. Ghys,
	\emph{Commutators and diffeomorphisms of surfaces},
	Ergodic Theory Dynam. Systems {\bf 24} (2004), no. 5, 1591--1617
\bibitem{Gratza}
	B. Gratza,
	\emph{Piecewise linear approximations in symplectic geometry},
	Diss. ETH No. 12499, 1998
\bibitem{Hatcher}
	A. Hatcher,
	\emph{Homeomorphisms of sufficiently large $P^2$-irreducible 3-manifolds},
	Topology {\bf 15} (1976), no. 4, 343--347
\bibitem{Hatcher_Smale}
	A. Hatcher,
	\emph{A proof of the Smale conjecture, $\Diff(S^3)\simeq O(4)$},
	Ann. of Math. (2) {\bf 117} (1983), no. 3, 553--607
\bibitem{Hurtado}
	S. Hurtado,
	\emph{Continuity of discrete homomorphisms of diffeomorphism groups},
	preprint; arXiv:1307.4447
\bibitem{Ivanov}
	N. Ivanov,
	\emph{Groups of diffeomorphisms of Waldhausen manifolds},
	Studies in topology II, LOMI {\bf 66} (1976), 172--176
\bibitem{Kapovich_RAAG}
	M. Kapovich,
	\emph{RAAGs in Ham},
	preprint, arXiv:1104.0348
\bibitem{Kim-Koberda}
	S.-H. Kim and T. Koberda,
	\emph{Anti-trees and right-angled Artin subgroups of planar braid groups},
	preprint; arXiv:1312.6465
\bibitem{Kopell}
	N. Kopell,
	\emph{Commuting Diffeomorphisms},
	in Global Analysis; Proc. Symp. Pure Math. XIV AMS, Providence, RI (1970), 165--184
\bibitem{Militon}
	E. Militon,
	\emph{\'El\'ements de distorsion du groupe des diff\'eomorphismes isotopes \`a
	l'identit\'e d'une vari\'et\'e compacte},
	preprint, arXiv:1005.1765
\bibitem{Thurston_stability}
	W. Thurston,
	\emph{A generalization of the Reeb stability theorem}, 
	Topology, {\bf 13} (1974), 347--352.
\bibitem{Wladis}
	C. Wladis,
	\emph{Thompson's group is distorted in the Thompson-Stein groups},
	Pacific J. Math. {\bf 250} (2011), no. 2, 473--485
\end{thebibliography}
\end{document}